\documentclass{amsart}

\usepackage{enumerate,comment}
\newtheorem{theorem}{Theorem}[section]

\newtheorem{lemma}[theorem]{Lemma}
\newtheorem{proposition}[theorem]{Proposition}
\newtheorem{corollary}[theorem]{Corollary}

\newtheorem{definition}[theorem]{Definition}

\DeclareMathOperator{\Var}{Var}

\DeclareMathOperator{\rado}{rado}

\title{Generating infinite random graphs}
\author[Bir\'o]{Csaba Bir\'o}
\author[Darji] {Udayan B.~Darji}
\email[Darji]{ubdarj01@louisville.edu}
\email[Bir\'o]{csaba.biro@louisville.edu}
\address{Department of Mathematics, University of Louisville, Louisville, KY 40292, USA}
\begin{document}

\begin{abstract}
We define a growing model of random graphs. Given a sequence of nonnegative
integers $\{d_n\}_{n=0}^\infty$ with the property that $d_i\leq i$, we construct a
random graph on countably infinitely many vertices $v_0,v_1\ldots$ by the
following process: vertex $v_i$ is connected to a subset of
 $\{v_0,\ldots,v_{i-1}\}$ of cardinality $d_i$ chosen uniformly at
random. We study the resulting probability space. In particular, we give a new
characterization of random graph and we also give probabilistic methods for constructing
infinite random trees. 
\end{abstract}

\maketitle

\section{Introduction}

Consider the vertex set ${\mathbb N}$. Let $0 < p <1$ be fixed. For each pair
of distinct integers $n, m \in {\mathbb N}$, put an edge  between $n$ and $m$
with probability $p$. Let $G$ be the resulting graph on ${\mathbb N}$. A
classical 1963 Erd\H os--R\'enyi theorem \cite{Erd-Ren-63} states that with
probability one, any two such graphs are isomorphic, i.e., there is essentially
one random graph on ${\mathbb N}$.

In 1964 Rado \cite{Rad-64} gave an explicit construction of a graph $R$ which
is universal for the collection of all countable graphs. More precisely, he
showed that if $G$ and $H$ are any countable graphs and $\phi:G \rightarrow H$
a graph homomorphism, then there are embeddings $e_G:G \rightarrow R$, $e_H:H
\rightarrow R$ and a graph homomorphism $\psi: R \rightarrow R$ such that
$e_H^{-1}\circ\psi \circ e_G = \phi$, i.e., $R$ contains a copy of every
countable graph and every graph homomorphism between countable graphs can be
lifted to a graph homomorphism of $R$. 

The constructions of Erd\H os--R\'enyi and Rado seem very different but they
result in the same graph.  The reason for this is that both graphs satisfy the
following property: if $(A,B)$ are disjoint, finite sets of vertices, then
there are infinitely many vertices $v$ such that there is an edge between $v$ and
every element of $A$ and there are no edges between $v$ and any element of $B$.
It can be shown by back and forth method that any two graphs with above
property are isomorphic to each other. A graph with this property is often
called the Erd\H os--R\'enyi graph, the Rado graph or simply the random graph. 

Although the Rado graph is unique and the above definition is very simple, the
Rado graph has a rich structure and it enjoys attention from mathematicians
working in various camps. For example Cameron \cite{Cam-01} gave a number
theoretic description of this graph similar to that of the Paley graph. This
graph also enjoys attention from model theorists as it is an example of an
$\aleph_0$-categorical Fra\"iss\'e limit of finite structures.  Truss \cite{Tru-85}
initiated the group theoretic study of the group of automorphisms of the Rado
graph. We refer the reader to the survey paper of Cameron \cite{Cam-01} for
further interesting properties of the random graph along this direction.
Connections between percolation theory and the random graphs can be found in
the survey paper of  van der Hofstad \cite{Hof-10}. That Rado graph can be topologically
2-generated with a great deal flexibility was shown by the second author and Mitchell
\cite{Darji-Mitchell}.

Inspired by the construction of Erd\H os and R\'enyi, we 
introduced a procedure that is flexible enough to generate a large class of
infinite graphs, and essentially generalizes the Erd\H os--R\'enyi process. We add only finitely many incident vertices to each vertex,
determined by a given sequence.

More rigorously, suppose a sequence of integers $\{d_i\}_{i=0}^\infty$ is
given, with the property $0\leq d_i\leq i$ for all $i$. Let
$V=\{v_0,v_1,\ldots\}$ be a set of vertices.  For $i=0,1,\ldots$, in round
$i$, we first choose $A\subseteq \{v_0,\ldots,v_{i-1}\}$  of cardinality $d_i$ with
the uniform distribution on the set of all sets of size $d_i$ subsets of $
\{v_0,\ldots,v_{i-1}\}$. Then, we add edges $v_iu$ for all $u\in A$. The result is
a random graph on countably many vertices. We strive to understand the
resulting probability space, in particular, we would like to determine the
atoms (graphs with positive probability), and cases when there is only one atom
with probability $1$. In this latter case, we say that the probability space is
\emph{concentrated}.

\subsection{Related literature}

\subsubsection{Preferential attachment models}

In the preferential attachment models, the new vertex is adjacent to earlier
vertex or vertices with a probability that depends on the current
degree of an existing vertex. One of the first examples of empirical study of
this model is by Barab\'asi and Albert \cite{Bar-Alb-99}, and a rigorous
mathematical framework was defined by Bollob\'as and Riordan \cite{Bol-Rio-04}.
These, and subsequent works study large finite graphs as opposed to the
limiting behavior.

An infinite version of the preferential attachment model (for multigraphs) was studied by
Kleinberg and Kleinberg \cite{Kle-Kle-05}. In their paper, the sequence $d_i$ is
constant. Since the preferential attachment model is substantially different
from ours, they get very different results, but some of the techniques they use
are similar to ours.

\subsubsection{Copying models}

This model was first introduced by Kumar et al.\
\cite{Kum-Rag-Raj-Siv-Tom-Upf-99}, and later, a slightly modified and
generalized version was defined by Bonato and Janssen \cite{Bon-Jan-07}. In
their construction, besides the sequence $d_i$, an initial finite graph $H$ and a probability $p\in[0,1]$
is given.
\begin{itemize}
\item Let $G_0=H$.
\item To construct $G_i$, add a new vertex $v$ to $G_{i-1}$ and choose its neighbors as follows.
\begin{itemize}
\item Choose a vertex $u\in V(G_{i-1})$ uniformly at random (called the
\emph{copy vertex}). Connect $v$ to each neighbor of $u$ with probability $p$.
\item Choose a set of $d_i$ vertices from $V(G_{i-1})$ uniformly at random, and
connect $v$ to each vertex in this set.
\item Delete multiple edges if necessary.
\end{itemize}
\end{itemize}
Clearly, our process is a special case of this, when $p=0$.

In \cite{Bon-Jan-07}, the authors only study the case when $d_i=\Theta(i^s)$ for
some $s\in [0,1)$. Although we study very similar models, the common
special case of our theorems is quite narrow: we imposed the extra condition
that $p=0$, and they imposed strong extra conditions on $d_i$.
Nevertheless, it is interesting to note that
for the narrow special case when our assumptions coincide, our
Theorem~\ref{thm:rado} implies the conclusions of their main theorems
(Theorems~2.2 and 2.3 in \cite{Bon-Jan-07}), and more.

\subsubsection{The Janson--Severini process}

Janson and Severini
\cite{Jan-Sev-13} introduced a process that also includes ours. Their construction is the following.
For all $i=1,\ldots$, let $\nu_i$ be a probability distribution on
$\{0,1,\ldots,i\}$. Construct the random graph $G_i$ as follows.
\begin{itemize}
\item Let $G_0=K_1$, the graph on a single vertex.
\item Let $D_i$ be a random variable with distribution $\nu_i$, and construct
$G_i$ by adding a new vertex to $G_{i-1}$ and connecting it to a uniformly
random subset of size $D_i$ of $V(G_{i-1})$.
\end{itemize}

Of course our model is the special case of theirs when $\nu_i$ is a point mass
at $d_i$. In fact, as an application of our theorems, we venture out to prove
certain limiting behavior in their general model (which we call the ``double
random process'') in Corollary~\ref{cor:dblrnd1}, and
Theorem~\ref{thm:dblrnd2}. However, unlike us, they study the graphons, as
limits of their sequence. (Graphons were introduced by Lov\'asz and Szegedy
\cite{Lov-Sze-06} and Borgs, Chayes, Lov\'asz, S\'os and Vesztergombi
\cite{Bor-Cha-Lov-Sos-Ves-08}.) In their main theorem, they determine the limit graphon when $D_n/n\overset{\text{p}}\to\nu$ for some probability measure $\nu$ on $[0,1]$.

\section{Summary and outline}

In Section~\ref{s:nonconcentrated} we discuss some minor results. We quickly
show how different this model is from the Erd\H os--R\'enyi model in that it
can easily result in non-concentrated spaces.

The main discussion starts in Section~\ref{s:rado}. The paper
contains two major results. In Section \ref{s:rado}, we prove the first one (stated in this section as Theorem~\ref{thm:rado}), which was
motivated by the effort of characterizing the sequences that will a.s.\ result
in the Rado graph. We did more than that: we defined a degree of similarity of
a graph to the Rado graph, and we can determine from the sequence how similar
the resulting graph will be to the Rado graph.

\begin{definition}
Let $G$ be a graph and $A,B\subseteq V(G)$. We say that a vertex $v$ is a
\emph{witness} for the ordered pair $(A,B)$, if $v$ is adjacent to every vertex
in $A$, and $v$ is not adjacent to any vertex in $B$.
\end{definition}

\begin{definition}
Let $G$ be a graph. For a nonnegative integer $k$, we say that $G$ is \emph{$k$-Rado}, if  every pair of disjoint
sets of vertices $(A,B)$ with $|A| \le k, |B| \le k$ has infinitely many witnesses.

The number $\rado(G)=\sup\{k: \text{$G$ is $k$-Rado}\}$ is the \emph{radocity} of $G$.
\end{definition}

Clearly every graph is $0$-Rado, and if a graph is $k$-Rado, it is also
$k'$-Rado for all $k'<k$. Also, by the Erd\H os--R\'enyi Theorem, $G$ is
isomorphic to the Rado graph if and only if $\rado(G)=\infty$.

We note that the definition of a witness is not new. Clearly Erd\H os and
R\'enyi knew about the property, and the same language is used by Spencer in
the book \cite{Spe-TSLORG}. Similar properties for a graph to be $k$-Rado also
appeared in the literature. Still in Spencer's book, the property $A_{r,s}$ is
defined: a graph satisfies the property $A_{r,s}$, if every pair of disjoint
sets of vertices $(A,B)$ with $|A|=r$, $|B|=s$ has a witness. Note the major
difference that $A_{r,s}$ requires only one witness, while $k$-Rado requires
infinitely many witnesses, so e.g.\ a double ray has $A_{1,1}$, but it is not $1$-Rado.

Another similar property is called $n$-e.c.\ (see e.g.\cite{Bon-09}). A graph
has this property, if every pair of disjoint sets of vertices $(A,B)$ with
$|A\cup B|=n$ has a witness. So a graph has $n$-e.c.\ if and only if it has
$A_{r,s}$ for all $r+s=n$.

Finally, Winkler used the colorful term \emph{Alice's Restaurant property} for
a graph that is $k$-Rado for all $k\geq 0$, in other words, the radocity of the
graph is $\infty$. As mentioned above, this happens if and only if the graph is
the Rado graph.

In Section~\ref{s:rado} we will prove the following theorem. It shows that the
radocity of the graph is determined by the sequence, not by the random process.
In the statement, and throughout the paper, we will use the standard notation
$n_{(k)}=n(n-1)\ldots(n-k+1)$ with $n_{(0)}=1$ (even if $n=0$). In addition, we
define $0^0=1$ if this power appears as a term of a series.

\begin{theorem}\label{thm:rado}
As before, let $\{d_i\}$ be such that $0\le d_i\le i$. Let
\begin{align*}
k_1&=\sup\left\{t\in\mathbb{N}:\sum_{n=1}^\infty\left(\frac{d_n}{n}\right)^t\left(\frac{n-d_n}{n}\right)^t=\infty\right\},\\
k_2&=\sup\left\{t\in\mathbb{N}:\sum_{n=1}^\infty
\frac{(d_n)_{(t)}(n-d_n)_{(t)}}{(n)_{(2t)}}=\infty\right\}.
\end{align*}
Then $k_1=k_2$, and the process a.s.\ generates a graph of radocity $k_1$ (and $k_2$). 
\end{theorem}

As a corollary, we achieve our original motivation.

\begin{corollary}\label{cor:min}
Let $a_n=\min\{\frac{d_n}{n},\frac{n-d_n}{n}\}$.
\begin{enumerate}[(i)]
\item If $\sum_{n=1}^\infty a_n^k$ diverges for all positive integers $k$, then 
the process almost surely generates the Rado graph.
\item If there is a positive integer $k$ for which $\sum_{n=1}^\infty a_n^k$
converges, then the process almost surely does not generate the Rado
graph.
\end{enumerate}
\end{corollary}

This also shows that our result is essentially a generalization of the result of
Erd\H os and R\'enyi. See Section~\ref{s:rado} for more details.

\subsection{Examples}

In the following examples, to avoid clutter, we will omit floor and ceiling signs.

\begin{itemize}
\item If $d_n=n/2$, then $\rado(G)=\infty$.
\item If $0<c<1$, and $d_n=cn$, then $\rado(G)=\infty$.
\item If $d_n=\sqrt{n}$, then $\rado(G)=2$.
\item If $d_n>0$ is constant, then $\rado(G)=1$.
\item If $k\geq 1$ integer, and $d_n=n^{(k-1)/k}$, then $\rado(G)=k$.
\item If $d_n=\log n$, then $\rado(G)=1$.
\end{itemize}

In Section~\ref{s:zeroone} we focus on $0$--$1$ sequences. From the discussion
above it is clear that the resulting graphs will a.s.\ have radocity $0$ or
$1$, but we aim to describe the random graph in more details.

Recall that the probability space is \emph{concentrated} if there exists a
graph $G$ such the process generates a graph isomorphic to $G$ with probability
$1$. A graph $G$ is an \emph{atom} of the space if the process generates a
graph isomorphic to $G$ with positive probability.

To state a compact theorem we introduce some elaborate notation to denote
certain infinite graphs.  Let $T$ be a finite tree.  Let $F_T$ be the forest
that consists of infinitely many copies of $T$, as components. Let $F_n=\bigcup\{F_T: T\text{ is a tree of size }n\}$. Note that $F_1$ is
the countably infinite set with no edges, and $F_2$ is the countably infinite
matching.

We will also use the term \emph{$\omega$-tree} for the unique countably
infinite tree in which every vertex is of infinite degree.

\begin{theorem}\label{thm:zo}
Suppose $d_n\in\{0,1\}$ for all $n\in\mathbb{N}$.
\begin{enumerate}[i)]
\item If $\sum_{i=1}^\infty\frac{d_i}{i}=\infty$, then the space is
concentrated, and the atom is a graph whose components are $\omega$-trees, and
the number of components is equal to the number of zeroes in the sequence.\label{thmzo:dense}
\item Suppose $\sum_{i=1}^\infty\frac{d_i}{i}<\infty$. Let
$t_n=\sum_{i=n}^\infty \frac{d_i}{i}$, and $k=\min\{\kappa\geq 2: \sum_l d_l
t_{l+1}^{\kappa-2}<\infty\}$. (We  set $k=\infty$ if the set in question is empty). The space has infinitely many atoms, and all of them
are of the form $F\cup\left[\bigcup_{i<k}F_i\right]$ where $F$ is some finite forest.\label{thmzo:sparse}
\end{enumerate}
\end{theorem}

Even though this theorem is not a complete description of the probability space,
it describes completely what the atoms are. The distinction of the sequences in
part~\ref{thmzo:sparse}) is extremely subtle, and the proof is very elaborate.
Nevertheless, we strived for clarity, and we divided the whole proof into small
lemmas, so by the time we are ready to prove the theorem, we can use the
machinery that will have been built up.

This theorem is the other major result of the paper, and arguably the more
difficult one.

\section{Non-concentrated spaces}\label{s:nonconcentrated}
It would perhaps be not completely na\"\i ve to think that something similar happens here
as in the Erd\H os--R\'enyi model. In this section we demonstrate that is
far from being correct. Therefore, we will show examples of non-concentrated
spaces.

The following proposition is actually about a very simple example of
\emph{concentration}, but we will use it as a tool to show non-concentration in
some other cases.

\begin{proposition}
The sequence $0,1,1,1,\ldots$ a.s.\ generates the $\omega$-tree.
\end{proposition}

\begin{proof}
We will prove a more general statement later, see Theorem~\ref{thm:zo}.
\end{proof}

\begin{corollary}
Consider a sequence of the form $d_0,d_1,\ldots,d_k,1,1,1,\ldots$. Let
$G_1,\ldots,G_l$ be the set of finite nonisomorphic graphs on $v_0,\ldots,v_k$
that is possible to be generated by the process using $d_1,\ldots,d_k$. For
each $i$, let
$G_i'$ be the graph constructed from $G_i$ by
attaching an $\omega$-tree to every vertex. Then the graphs $G_1',\ldots,G_l'$
are the atoms of the space, with the probabilities are inherited from the
finite part of the process.
\end{corollary}

The corollary above shows that it is easy to construct a sequence whose
associated probability space is not concentrated. E.g. 0,1,2,1,2,1,1,1,\dots
However these examples are very special in the sense that they are
eventually all 0's and 1's, so after that point no more cycles are generated.
Nevertheless, the following proposition shows that non-concentrated probability
spaces can be found for other kind of sequences.

\begin{proposition}
There exists a sequence $\{d_i\}$ with non-concentrated
probability space such that for all positive integers $N$ there exists $n>N$
such that $a_n\neq 0$ and $a_n\neq 1$. 
\end{proposition}

\begin{proof}
We will construct a sequence consisting mostly of 1's, but infinitely many 2's
inserted. The sequence starts with $0,1,1,2$. We set $p_0=2/3$, and we note
that $p_0$ is the probability that the first $4$ vertices include a triangle.
Then let $k$ be the least integer such that $k/\binom{k}{2}<3/4-p_0$. Set
$d_4=\cdots=d_{k-1}=1$, and $d_k=2$. Note, that the probability that a
triangle is generated by $v_k$ is $p_1:=k/\binom{k}{2}$. In general, after
the $l$th $2$ in the sequence, let $k$ be a sufficiently large integer for
which $d_k$ is not yet defined and
\[
\frac{k+l-1}{\binom{k}{2}}<\frac{3}{4}-\sum_{i=0}^{l-1}p_i.
\]
Set $d_k=2$ and set all the elements before
$d_k$ that are not yet defined to be $1$.
Note that the probability that a triangle is generated at
$v_k$ equals $p_l=\frac{k+l-1}{\binom{k}{2}}$.
Let $X$ be the random variable that denotes the number of triangles eventually
generated in $G$. Due to the linearity of expectation,
\[
\mu=E[X]=  \sum_{i=0}^\infty p_i.
\]
Clearly, from the definition of the sequence $2/3\leq \mu\leq 3/4$. That means
that
\[
\Pr[X=0]>0\quad\text{and}\quad\Pr[X>0]>0.
\]
The sets $[X=0]$ and $[X>0]$ partition the probability space, and neither of
them are of measure $0$, so the space can not be concentrated.
\end{proof}

\section{The Rado graph}\label{s:rado}

This section contains the proof of Theorem~\ref{thm:rado}, and
Corollary~\ref{cor:min}, with some additional discussion of some consequences.
We will make a frequent use of the following basic fact relating infinite products to infinite sums.
\begin{proposition}
Let $\{b_i\}_{i=0}^{\infty}$ be a sequence of real numbers such
that $0 < b_i <1$ and $\{d_i\}$ be a sequence of nonnegative integers. Then,
\[
0 < \prod_{i=1}^{\infty} (1-b_i)^{d_i}  \iff \sum _{i=1}^{\infty} d_ib_i<  \infty
\]
\end{proposition}

We start with a simple technical lemma.

\begin{lemma}\label{lemma:limitcmp}
Fix a nonnegative integer $k$. The infinite series
\[
\sum_{n}^\infty\left(\frac{d_n}{n}\right)^k\left(\frac{n-d_n}{n}\right)^k
\qquad\text{and}\qquad
\sum_{n}^\infty\frac{(d_n)_{(k)}(n-d_n)_{(k)}}{(n)_{(2k)}}
\]
either both converge or both diverge.
\end{lemma}
\begin{proof}
If $k\leq 1$ then the statement is trivial. If $k\geq 2$, then partition the
terms into three parts: $A=\{i: d_i<k\}$, $B=\{i:n-d_i<k\}$, and
$C=\mathbb{N}\setminus (A\cup B)$. It is clear that over the terms indexed by
$A$ and $B$, both series converge, so the behavior is decided by the terms over
$C$. For those we use a generalized limit comparison test, and show that the
lim inf and lim sup of the ratio of the terms are positive and finite.

To see this last statement, notice that
\[
1\leq\frac{d_n}{d_n},\frac{d_n}{d_{n-1}},\ldots,\frac{d_n}{d_n-k+1}\leq k
\]
so
\[
1\leq\frac{(d_n)^k}{(d_n)_{(k)}}\leq k^k.
\]
A similar statement can be made about $\frac{(n-d_n)^k}{(n-d_n)_{(k)}}$, so we
see that the lim inf of ratio of the terms is at least $1$, and the lim sup is
at most $k^{2k}$.
\end{proof}

\subsection{Proof of Theorem~\ref{thm:rado}}

Note that $k_1=k_2$ is a consequence of Lemma~\ref{lemma:limitcmp}. We will denote this number by $k$, and we will go back and forth between its two equivalent
definitions at our convenience.

Now we prove that the graph generated is almost surely $k$-Rado.

The statement is trivial for $k=0$. Let $A,B$ be two finite disjoint vertex sets with $|A|=|B|=k\geq 1$, and let
$N$ be a positive integer. It is sufficient
to show that the pair $(A,B)$ has a witness with probability $1$ among the
vertices $v_N,v_{N+1},\ldots$.

For a given vertex $v_n$, let $p_n$ be the probability that $v_n$ is a witness
for $(A,B)$.
Now pick a vertex $v_n$ such that $n>\max\{i:v_i\in A\cup B\}$ and $n\geq N$.
Then
\[
p_n=\frac{\binom{n-2k}{d_n-k}}{\binom{n}{d_n}}
=\frac{(d_n)_{(k)}(n-d_n)_{(k)}}{(n)_{(2k)}}.
\]
Note that this holds whether $d_n\geq k$ or $d_n<k$; in the latter case $p_n=0$.
Hence
$\sum_{n=N}^\infty p_n$ diverges, and
then $\prod_{n=N}^\infty(1-p_n)=0$, which is the probability that the
pair $(A,B)$ has no witness beyond (including) $v_N$.

It remains to be proven that if $k<\infty$, then a.s \ the graph is not
$k+1$-Rado. It suffices to prove that there is a pair $(A,B)$ of finite
disjoint vertex sets with $|A|=|B|=k+1$ such that a.s.\ $(A,B)$ has finitely
many witnesses. Indeed, we prove that this is the case for every such 
pair of vertex sets $(A, B)$. To obtain a contradiction suppose that this is not true: 
that is there are disjoint sets $A, B$ of
vertices with $|A|=|B|=k+1$ and the probability that $(A,B)$ has finitely many
witnesses is $p <1$.  Let $q_N$ be the probability
that $(A,B)$ has no witness beyond (including) $v_N$. We note
that \begin{equation}\label{eq:contra}
q_N\leq p\text{ for all }N.
\end{equation}

On the other hand, similarly as above, the probability that a given vertex
$v_n$ is a witness for $(A,B)$ (if $n$ is large enough) is
\[
p_n=\frac{(d_n)_{(k+1)}(n-d_n)_{(k+1)}}{(n)_{(2(k+1))}}.
\]
This time, we know that $\sum p_n<\infty$, so $\prod (1-p_n)>0$. Hence there
exists $N$ such that
\[
q_N=\prod_{n=N}^\infty(1-p_n)>p.
\]
But this contradicts (\ref{eq:contra}).
\qed

\subsection{Proof of Corollary~\ref{cor:min}}
Suppose that $\sum_{n=1}^\infty a_n^k$ diverges for all $k$. Since
\[
\sum_{n=1}^\infty\left(\frac{d_n}{n}\right)^k\left(\frac{n-d_n}{n}\right)^k
\geq\sum_{n=1}^\infty a_n^{2k},
\]
we get that
\[
\sum\left(\frac{d_n}{n}\right)^k\left(\frac{n-d_n}{n}\right)^k
\]
diverges for all $k$, and therefore we get a.s.\ $\rado(G)=\infty$.

Now suppose that there is a positive integer $k$ for which $\sum_{n=1}^\infty a_n^k$
converges. Since
\[
a_n^k\geq\left(\frac{d_n}{n}\right)^k\left(\frac{n-d_n}{n}\right)^k
\geq
\prod_{i=0}^{k-1}\frac{d_n-i}{n}\cdot\frac{n-d_n-i}{n},
\]
we have that for large enough $n_0$,
\begin{multline*}
\sum_{n=n_0}^\infty a_n^k\cdot 2^{2k}
\geq
\sum_{n=n_0}^\infty\left(
\prod_{i=0}^{k-1}\frac{d_n-i}{n}\cdot\frac{n-d_n-i}{n}
\prod_{i=0}^{2k-1}\frac{n}{n-i}
\right)\\
=
\sum_{n=n_0}^\infty\frac{(d_n)_{(k)}(n-d_n)_{(k)}}{(n)_{(2k)}},
\end{multline*}
and therefore the last sum converges. Thus the graph a.s.\ has finite radocity.
\qed

\begin{corollary}\label{cor:limsup}
Let $a_n=\min\{\frac{d_n}{n},\frac{n-d_n}{n}\}$. If $\limsup a_n>0$, then the process a.s.\ generates the Rado graph.
\end{corollary}

\begin{proof}
Direct consequence of Corollary~\ref{cor:min}.
\end{proof}

The \emph{double random process} is when we even chose the sequence in random,
choosing $d_i$ with some distribution from the interval $[0,i]$. Note that
Janson and Severini \cite{Jan-Sev-13} study the double random process from a
different point of view. The following corollary states that in some sense,
almost all double random processes will result in the Rado graph.

\begin{corollary}\label{cor:dblrnd1}
If there exist $\epsilon>0$, $p_0>0$, and $M$ integer such that for $n>M$, $\Pr[\epsilon
n\leq d_n\leq (1-\epsilon)n]\geq p_0$, then the double random process
a.s.\ generates the Rado graph.
\end{corollary}

\begin{proof}
It is easy to see that Corollary~\ref{cor:limsup} is a.s.\ satisfied.
\end{proof}

\section{Density, sparsity, degrees, and stars}\label{s:density}

The main goal of this section is to analyze how certain ``density'' conditions
on the sequence will affect the resulting graph. One important result from this
section (Theorem~\ref{thm:degrees}) will also be used in
Section~\ref{s:zeroone} to analyze zero--one sequences.

For the rest of the section, we will use the notation $s_n=\sum_{i=0}^n d_i$,
the partial sum of the sequence $\{d_i\}$.

It will be useful to distinguish sequences based on convergence of certain
partial sums. When $\sum d_i/i=\infty$, we will refer to this situation as the
``dense'' case. The opposite case, when $\sum d_i/i<\infty$, will be called the
``sparse'' case. A subcase of the sparse case, when even $\sum s_i
d_i/i<\infty$, will be called the ``very sparse case''.

We begin with a simple proposition on binomial coefficients.

\begin{proposition}\label{prop:combi}
Let $n,d,m\geq 0$ integers with 
$\frac{m}{n-d}\leq 1$. Then
\[
\left(1-\frac{m}{n-d}\right)^d
\leq
\frac{\binom{n-m}{d}}{\binom{n}{d}}
\leq 
\left(1-\frac{m}{n}\right)^d.
\]
\end{proposition}
\begin{proof}We note that
\[
\frac{\binom{n-m}{d}}{\binom{n}{d}}=\frac{(n-m)_{(d)}}{(n)_{(d)}}=\prod_{i=0}^{d-1}\left(1-\frac{m}{n-i}\right).
\]
Then bound the product by replacing all factors with the largest factor, and
then with the smallest factor to obtain the desired inequality.
\end{proof}

\begin{lemma}\label{lemma:hit}
Let $v_k$ be a vertex.
\begin{enumerate}[i)]
\item If $\sum d_i/i=\infty$ then for all $N>k$, a.s.\ $v_k$ has a neighbor
beyond
$v_N$.
\item If $\sum d_i/i<\infty$ then there exists $M$ such that with positive
probability $v_k$ has no neighbor beyond $v_M$; furthermore, for all $\epsilon>0$
there exists an $M'\geq M$ such that $\Pr(v_k\text{ has a neighbor beyond
}v_{M'})<\epsilon$.
\end{enumerate}
\end{lemma}

\begin{proof}
Let $E_l$ be the event that $v_k$ has no neighbor beyond $v_l$. We will estimate
the probability of $E_l$. For any $i>k$, we have
$\Pr(v_i\not\sim v_k)={\binom{i-1}{d_i}}/{\binom{i}{d_i}}$, so
$\Pr(E_l)=\prod_{i=l}^\infty {\binom{i-1}{d_i}}/{\binom{i}{d_i}}$.

Using
Proposition~\ref{prop:combi}, we have
\[
\Pr(E_l)\leq
\prod_{i=l}^\infty\left(1-\frac{1}{i}\right)^{d_i}.
\]
If $\sum d_i/i=\infty$, then the product
on the right hand side is zero. Thus, we have that $\Pr(E_l)=0$ for all
$l>i_k$, which proves the first part. 

If $\sum
d_i/i<\infty$, then there exists an $M$ such that for all $i\geq M$, $d_i/i\leq
1/2$,
and $1/(i-d_i)\leq 1$. Then we may use the other part of Proposition~\ref{prop:combi} to get
\begin{equation}
\Pr(E_M)\geq\prod_{i=M}^\infty
\left(1-\frac{1}{i-d_i}\right)^{d_i}.\label{eq:lower}
\end{equation}
Also, in this case,
\[
\sum_{i=M}^\infty \frac{d_i}{i-d_i}=\sum_{i=M}^\infty
\frac{1}{1-d_i/i}\cdot\frac{d_i}{i}\leq\sum_{i=M}^\infty 2\cdot\frac{d_i}{i}<\infty,
\]
so $\Pr(E_M)>0$.

The last statement follows from the fact that the right hand side in
(\ref{eq:lower}) is positive, therefore its tail end converges to $1$, so for
all $\epsilon>0$ there
exists $M'$ for which $\Pr(E_{M'})>1-\epsilon$.
\end{proof}

\begin{theorem}[Density and degrees]
\label{thm:degrees} The following statements hold.
\begin{enumerate}[i)]
\item If $\sum d_i/i=\infty$, then the process a.s.\ generates a graph in
which each vertex is of infinite degree.\label{thm:i}
\item If $\sum d_i/i<\infty$, then the process a.s.\ generates a graph in which each vertex is of finite degree.\label{thm:ii}
\end{enumerate}
\end{theorem}

\begin{proof}
Both parts follow from Lemma~\ref{lemma:hit}. Let $v_k$ be a vertex
and let $N>k$ be an integer. In case~\ref{thm:i}), Lemma~\ref{lemma:hit} implies that a.s.\ $v_k$ has a neighbor beyond
$v_{N}$. This being true for arbitrary $N>k$, we conclude that a.s.\ $v_k$ has
infinitely many neighbors.

In case~\ref{thm:ii}), the lemma provides that the probability that $v_k$ is of
infinite degree is less than $\epsilon$ for all $\epsilon>0$, and therefore
that probability is $0$.
\end{proof}

Recall that the bipartite graphs $K_{1,l}$ for $l=0,1,2,\ldots$ are called
\emph{stars}. (For convenience, we allow $l=0$. In this case,  $K_{1,l}$ is
simply a singleton set.) To emphasize the size of the star, $K_{1,l}$ will
often be called an $l$-star. We say that a vertex in a graph is \emph{in a
star,} respectively \emph{in an $l$-star}, if the connected component of the
vertex is a star, respectively an $l$-star.

\begin{lemma}\label{lemma:allstar}
Suppose $\sum s_i d_i/i<\infty$. Then the process a.s.\ generates a 
graph $G$ which has the property that there exists an $N_1=N_1(G)$ such that for
all $n\geq N_1$ with $d_n >0$, the vertex $v_n$ will attach back to vertices
with current degree $0$. More rigorously, the vertex $v_n$ has the property
that if $v_j\sim v_n$, and $j<n$, then
$v_j$ has no neighbor before $v_n$.
\end{lemma}

\begin{proof}
Note that $\sum s_i d_i/i<\infty$ implies $s_nd_n/n\to 0$ as $n\to\infty$, so there
exists $N$ such that for all $n>N$, $s_nd_n/n<1/3$. Consider such an $n$. Below
we will compute the probability that at stage $n$, the vertex $v_n$ attaches only
to the vertices that are currently of degree 0, i.e., singletons.

Observe that during the process, for every $i$ with $d_i=0$ one singleton is
created, and if $d_i>0$, then at most $d_i$ singletons are destroyed. One can
view this as always creating a singleton and then destroying no more than $2d_i$. So at step $i$, the number of singletons is at least $i-\sum_{j=0}^i
2d_j=i-2s_i$, and then, by Proposition~\ref{prop:combi} and the fact that
$d_n/n < 1/3$, we obtain that the probability 
that $v_n$ attaches to only singletons is
at least
\begin{eqnarray*}
\frac{\binom{n-2s_n}{d_n}}{\binom{n}{d_n}}&
\geq & \left ( 1-\frac{2s_n}{n-d_n} \right)^{d_n}
\geq \left ( 1-\frac{2s_n/n}{1-d_n/n} \right)^{d_n}\\
&\geq & \left ( 1-\frac{2s_n/n}{2/3} \right)^{d_n}
=  \left ( 1-\frac{3s_n}{n} \right)^{d_n}.
\end{eqnarray*}
Hence the probability that this happens to all vertices beyond $N$ is at least
\[
\prod_{n=N}^\infty \left( 1-\frac{3s_n}{n} \right)^{d_n}.
\]

The last product is positive as $\sum s_i d_i/i<\infty$. Hence, we have that for all $\epsilon>0$
there exists $M$, such that $\prod_{n=M}^\infty\left(1-\frac{3s_n
d_n}{n}\right)>1-\epsilon$. That means that with probability greater than
$1-\epsilon$, every vertex beyond $M$ attaches to singletons.

To complete the proof suppose that the existence of an
$N_1$ as in the statement has probability $p<1$. Choose $\epsilon<1-p$.
According to the argument above, there exists an $M$ such that the probability
that every vertex beyond $M$ attaches to singletons is greater than
$1-\epsilon>p$, and since $M$ is a suitable choice for $N_1$, this is a
contradiction.
\end{proof}

\begin{theorem}[Very sparse case]
Suppose $\sum s_i d_i/i<\infty$. Then the process a.s.\ generates a graph $G$
for which there is $N=N(G)$ such that for all $n >N$, $v_n$ is in a star. Moreover,
in addition, if $d_n >0$, then $v_n$ is in a $d_n$-star. 
\end{theorem}

\begin{proof}
By Lemma~\ref{lemma:allstar}, the process a.s.\ generates a graph $G$ for
which there is $N=N(G)$ such that for all $n > N$, at stage $n$, either $d_n=0$
or $v_n$ attaches to $d_n$ many
current degree $0$ vertices which precede $v_n$. Note that $v_k$, $k >n$, leaves untouched the star generated
by $v_n$.  Therefore, we obtain that $v_n$ is in a $d_n$-star. If $n > N$
and $d_n=0$, then the component of $v_n$ in $G$ is either a singleton
or a star generated by some $v_m$, $m>n$. \end{proof}

\section{Zero-one sequences}\label{s:zeroone}

As the title suggests, the standing assumption for the section is that $0\leq
d_n\leq 1\text{ for all }n\geq 0$.  We will also assume that there are
infinitely many $1$'s in the sequence, for otherwise we really have a finite
sequence and an essentially finite graph (plus isolated vertices), and we get a
problem of a very different flavor. It is clear that the number of connected
components of the generated graph is equal to the number of $0$'s in the
sequence, and each component is a tree.

\subsection{Notation}

For this section it will be convenient to introduce some notation
to denote certain tuples of indices and sums and products. First we introduce
notations on products and tuples.

We let $\mathbb{N} ^{<\mathbb{N}}$ denote set of all
finite strings of $\mathbb{N}$ including the empty string.
We define $f:\mathbb{N} ^{<\mathbb{N}} \rightarrow [0,\infty)$
by
\[f(\sigma) =\prod_{i=1}^n \frac{d_{\sigma_i}}{{\sigma_i}},
\]
where $\sigma=(\sigma_1,\ldots,\sigma_n)$. (By convention,
$f(\sigma) =1$ when $\sigma$ is the empty string.) The definition of
$f(\sigma)$ depends on the fixed sequence $\{d_n\}$.

Suppose $i, j, l \ge 1$ with $i \le j$. Then,
\begin{gather*}
A^l_{i} = \{(\sigma_1,\ldots,\sigma_l) \in \mathbb{N}^l|\min \{\sigma_1,\ldots,\sigma_l\}=i\},\\
B^l_{i} = \{\sigma \in A^l_i| \sigma \text{ is strictly increasing}\},\\
B^l_{i,j} = \{(\sigma_1,\ldots,\sigma_l) \in B^l_i|\max\{\sigma_1,\ldots,\sigma_l\} \le j\},\\
C^l_{i} = \{\sigma \in A^l_i| \sigma \text{ is injective}\},\\
D^l_i = \{i,i+1,\ldots\}^l=\{(\sigma_1,\ldots,\sigma_l) \in \mathbb{N}^l|\min \{\sigma_1,\ldots,\sigma_l\}\geq i\}.
\end{gather*}

The following notations are about sums and series.
\begin{gather*}
s_{m,n}=\sum_{i=m}^n d_i\qquad
s_n=s_{0,n}=\sum_{i=0}^n d_i\qquad
t_{n,m}=\sum_{i=n}^m \frac{d_i}{i}\qquad
t_n=t_{n,\infty}=\sum_{i=n}^\infty \frac{d_i}{i}.
\end{gather*}

\subsection{The sparse case for zero--one sequences}

What we proved in Section~\ref{s:density} essentially gives us the
behavior of the probability space in the dense case, when $\sum d_i/i=\infty$,
and the very sparse case, when $\sum s_id_i/i<\infty$. We will summarize these
findings (and much more) in Theorem~\ref{thm:zo}. This subsection will entirely
be devoted to the sparse case. Accordingly, throughout the subsection we will
assume that $\sum d_i/i<\infty$. Our findings \emph{will} apply for the very
sparse case, giving an alternative proof of the characterization of the space
in that case. But note, that the findings of Section~\ref{s:density}
apply \emph{in the general setting} (not only
zero--one), so those theorems still have their importance.

\begin{proposition}\label{lemma:limit}
$\lim_{i\to\infty}s_i/i=0$.
\end{proposition}

\begin{proof}
Let $\epsilon>0$ small. 
Then, 
\[
\frac{s_n}{n}
=\frac{\sum_{i=0}^n d_i}{n}
\leq\frac{\sum_{i=0}^{\lfloor \epsilon n/2\rfloor} d_i}{n}
	+\sum_{i=\lfloor \epsilon n/2\rfloor+1}^n \frac{d_i}{i}
\leq\frac{\epsilon}{2}+\sum_{i=\lfloor\epsilon n/2\rfloor+1}^\infty \frac{d_i}{i}.
\]
As $\sum d_i/i<\infty$, the second term converges to zero as $n\to\infty$, so
eventually it will be less than $\epsilon/2$.
\end{proof}

Let $a(k)_i$ denote the expected number of trees of size $k$ spanned by
the vertex set $\{v_0,\ldots,v_i\}$. We will prove a sequence of technical
lemmas about the sequences $a(k)$.

\begin{proposition}\label{prop:recursion} For $i\geq 1$ and $k\geq 2$,
\[
a(k)_i=a(k)_{i-1}+\frac{(k-1)a(k-1)_{i-1}}{i}d_i-\frac{k a(k)_{i-1}}{i}d_i.
\]
\end{proposition}

\begin{proof} Suppose $d_i=1$.
Consider the process just before we add the edge from $v_i$. Let $p_+$ be
the probability that we increase the number of trees of size $k$, and let $p_-$
be the probability that we decrease that number. In either case, the change is
$\pm 1$. Let $p_0=1-(p_+-p_-)$. On one hand,
$p_+=\frac{(k-1)a(k-1)_{i-1}}{i}$, and $p_-=\frac{k a(k)_{i-1}}{i}$. On the
other hand
\[
a(k)_i=p_+(a(k)_{i-1}+1)+p_-(a(k)_{i-1}-1)+p_0 a(k)_{i-1}=a(k)_{i-1}+(p_+-p_-).
\]

In the other case, if $d_i=0$, then $a(k)_i=a(k)_{i-1}$.
\end{proof}

\begin{lemma}\label{lemma:recursion}
Let $k\geq 2$. Then there exist positive constants $C_1, C_2$ such that for all $i\geq 1$
\[
C_1\sum_{j=k}^i\frac{a(k-1)_{j-1}}{j}d_j
\leq a(k)_i
\leq C_2\sum_{j=1}^i\frac{a(k-1)_{j-1}}{j}d_j.
\]
\end{lemma}

\begin{proof}
The upper bound is a straightforward consequence of
Proposition~\ref{prop:recursion}. Indeed, $a(k)_i\leq a(k)_{i-1}+\frac{k
a(k-1)_{i-1}}{i}d_i$, so $a(k)_i\leq k\sum_{j=1}^i\frac{a(k-1)_{j-1}}{j}d_j$.

For the lower bound, notice that
\[
a(k)_i\geq a(k)_{i-1}\left(1-\frac{kd_i}{i}\right)+\frac{a(k-1)_{i-1}}{i}d_i.
\]
This implies
\begin{multline*}
a(k)_i
\geq\sum_{j=1}^i \left[\frac{a(k-1)_{j-1}}{j}d_j\prod_{l=j+1}^i\left(1-\frac{k
d_l}{l}\right)\right]\\
\geq\sum_{j=k}^i \left[\frac{a(k-1)_{j-1}}{j}d_j\prod_{l=j+1}^i\left(1-\frac{k
d_l}{l}\right)\right]
\geq\sum_{j=k}^i \left[\frac{a(k-1)_{j-1}}{j}d_j\prod_{l=k+1}^\infty\left(1-\frac{k d_l}{l}\right)\right]\\
\geq Q\sum_{j=k}^i\frac{a(k-1)_{j-1}}{j}d_j,
\end{multline*}
where $Q=\prod_{l=k+1}^\infty\left(1-\frac{k d_l}{l}\right)$. Note that the
second inequality is correct, because for $j=1,\ldots,k-2$, we have
$a(k-1)_{j-1}=0$ (so the omitted terms are zero), and for $j=k-1$ the omitted term is nonnegative.
Also note that $\sum d_i/i<\infty$ implies $Q>0$.
\end{proof}

\begin{lemma}\label{lemma:upbd}
Let $k\geq 2$. Then there exists positive constants $K$ such that for all $i$,
\[
a(k)_i
\leq K\sum_{j=1}^i d_j t_{j+1}^{k-2}.
\]
\end{lemma}

\begin{proof}
We proceed by induction on $k$. Let $k=2$. By Lemma~\ref{lemma:recursion}, there exists $C_2$ such that $a(2)_i\leq
C_2\sum_{j=1}^i\frac{a(1)_{j-1}}{j}d_j\leq C_2\sum_{j=1}^i d_j$. 

Now suppose that $k\geq 3$. By Lemma~\ref{lemma:recursion} and the induction
hypothesis, there exist $C_2$ and $C$ positive constants such that
\begin{multline*}
a(k)_i
\leq C_2\sum_{j=1}^i\frac{a(k-1)_{j-1}}{j}d_j
\leq C\sum_{j=2}^i\sum_{l=1}^{j-1}d_lt_{l+1}^{k-3}\frac{d_j}{j}
\leq C\sum_{l=1}^{i-1}\sum_{j=l+1}^i d_l t_{l+1}^{k-3}\frac{d_j}{j}\\
\leq C\sum_{l=1}^{i-1}d_lt_{l+1}^{k-3}\sum_{j=l+1}^i\frac{d_j}{j}
\leq C\sum_{l=1}^{i-1}d_lt_{l+1}^{k-2}.
\end{multline*}
\end{proof}

\begin{lemma}\label{lemma:finitesum}
Let $l \ge 1$.
\[ \sum_{i=1}^{\infty} \sum_{\sigma \in A^l_i} f(\sigma) < \infty.
\]
\end{lemma}
\begin{proof} For $l=1$ the above statement is equivalent to the sparsity
condition $\sum d_i/i<\infty$.
For $l>1$, we note that 
\[
\sum_{\sigma \in A^l_i} f(\sigma)
\le\sum_{j=1}^l \sum_{\substack{\sigma \in A^l_i\\\sigma_j =i}}\frac{d_i}{i}\prod^l_{\substack{k=1\\ k \neq j }} \frac{d_{\sigma_k}}{{\sigma_k}}
= \sum_{j=1}^l \frac{d_i}{i}\sum_{\sigma \in D_i^{l-1}} f(\sigma)
=\sum_{j=1}^l \frac{d_i}{i}t^{l-1}_i  = l \frac{d_i}{i}t^{l-1}_i 
\]
Then, 
\[ \sum_{i=1}^{\infty} \sum_{\sigma \in A^l_i} f(\sigma) \le l  
\sum_{i=1}^{\infty} \frac{d_i}{i}t^{l-1}_i.
\]
As $\lim_{i \rightarrow \infty} t_i =0$, we have that $\{t_i\}_{i=1}^{\infty}$ is bounded
and hence the desired series converges.
\end{proof}

\begin{lemma}\label{lemma:bothdiverge}
Suppose that $l\ge 1$. Then, 
\[
\sum_{i=1}^{\infty} d_i t^l_i =\infty
\implies
\sum_{i=1}^{\infty}s_i \sum_{\sigma \in B^l_i} f(\sigma)=\infty.
\]
\end{lemma}
\begin{proof}
By rearranging and switching the order of summation, we have that
\begin{equation}
\label{star}
\infty=\sum_{i=1}^{\infty} d_i t^l_i
= \sum_{i=1}^{\infty} d_i \sum_{j=i}^\infty\sum_{\sigma\in A_j^l} f(\sigma)
= \sum_{j=1}^\infty\sum_{i=1}^j d_i\left(\sum_{\sigma\in A_j^l}
	f(\sigma)\right)
= \sum_{j=1}^{\infty}s_j \sum_{\sigma \in A^l_j} f(\sigma).
\end{equation}
We next observe that if $l=1$, then $A^l_i=B^l_i$ and the proof is complete. Hence,
let us assume that $l \ge 2$.  We will next show that
\begin{equation}
\label{starstar}\sum_{i=1}^{\infty}s_i \sum_{\sigma \in A^l_i \setminus C^l_i} f(\sigma) < \infty.
\end{equation}
\begin{multline*}
\sum_{i=1}^{\infty}s_i \sum_{\sigma \in A^l_i \setminus C^l_i} f(\sigma) \le \sum_{i=1}^{\infty}s_i   \sum_{1 \le j < k \le l } \sum_{\substack {\sigma \in A^l_{i}\\  \sigma_j=\sigma_k }} f(\sigma)\\
= \sum_{i=1}^{\infty}s_i   \sum_{1 \le j < k \le l } \left ( \sum_{\substack {\sigma \in A^l_{i}\\  \sigma_j=\sigma_k =i}} f(\sigma) + \sum_{m=i+1}^{\infty} \sum_{\substack {\sigma \in A^l_{i}\\  \sigma_j=\sigma_k =m}} f(\sigma) \right )\\
 \le \sum_{i=1}^{\infty}s_i\sum_{1\le j<k\le l}\left(\left(\frac{d_i}{i}\right)^2\sum_{\sigma \in D^{l-2}_{i}}f(\sigma) + \sum_{m=i+1}^{\infty}\sum_{\substack{1\le p\le l\\p\notin\{j,k\}}}\sum_{\substack {\sigma \in A^l_{i}\\  \sigma_j=\sigma_k =m\\ \sigma_p =i}} f(\sigma) \right )\\
= \sum_{i=1}^{\infty}s_i   \sum_{1 \le j < k \le l } \left (\left ( \frac{d_i}{i} \right )^2 \sum_{\sigma \in D^{l-2}_{i}} f(\sigma) + \sum_{m=i+1}^{\infty}  \sum_{\substack {1 \le p \le l \\  p \notin \{j,k\}}} \frac{d_i}{i} \frac{d_m}{m} \frac{d_m}{m} \sum_{\substack {\sigma \in D^{l-3}_{i}}} f(\sigma) \right )\\
\le \sum_{i=1}^{\infty}s_i   \sum_{1 \le j < k \le l } \left (  \left ( \frac{d_i}{i} \right )^2 \sum_{\sigma \in D^{l-2}_{i}} f(\sigma) +(l-2)\left ( \frac{d_i}{i} \right )^2\sum_{m=i+1}^{\infty} \frac{d_m}{m}\sum_{\substack {\sigma \in D^{l-3}_{i}}} f(\sigma) \right )\\
\le \sum_{i=1}^{\infty}s_i   \sum_{1 \le j < k \le l } \left (  \left ( \frac{d_i}{i} \right )^2 \sum_{\sigma \in D^{l-2}_{i}} f(\sigma) +(l-2)\left ( \frac{d_i}{i} \right )^2 \sum_{\substack {\sigma \in D^{l-2}_{i}}} f(\sigma) \right )\\
\le (l-1) \binom{l}{2} \sum_{i=1}^{\infty}s_i    \left ( \frac{d_i}{i} \right )^2  \sum_{\substack {\sigma \in D^{l-2}_{i}}} f(\sigma)
= (l-1) \binom{l}{2} \sum_{i=1}^{\infty} \frac{s_i}{i}  \frac{d_i}{i} t^{l-2}_i
< \infty.
\end{multline*}
As before, the last inequality follows as $\{s_i/i\}_{i=1}^{\infty}$  and $\{t_i\}_{i=1}^{\infty}$ are bounded sequences.

Putting (\ref{star}) and (\ref{starstar}) together we have that 
\[
\sum_{i=1}^{\infty}s_i \sum_{\sigma \in  C^l_i} f(\sigma)  = \infty. \]
Noting
\[
\sum_{i=1}^{\infty}s_i \sum_{\sigma \in  C^l_i} f(\sigma)  =l!\sum_{i=1}^{\infty}s_i \sum_{\sigma \in  B^l_i} f(\sigma),\]
the proof is complete.
\end{proof}

\begin{lemma}\label{lemma:lowertrees}
Let $k \ge 2$. Then, there exists $C$ and $N$ such that for all $i \in \mathbb{N}$, we have that
\[a(2)_i \ge C s_{N,i}, \mbox { and } 
\]
\[a(k)_i \ge C \sum_{j=3}^{i+3-k} s_{N,j-1} \sum _{\sigma \in B^{k-2}_{j, i}} f(\sigma)  \ \ \ \  \mbox{ for } k \ge 3.
\]
\end{lemma}
\begin{proof}
Let $N\geq 2$. Since $s_n/n\to0$, there exists $N$ such that for all $n>N$, $2s_n/n\leq1/2$.

First consider $k=2$. From Lemma~\ref{lemma:recursion} there exists a constant
$C$ such that for all $i\geq 2$
\[
a(2)_i\geq C\sum_{j=2}^i\frac{a(1)_{j-1}}{j} d_j.
\]
So for all $i\geq N$
\begin{multline*}
a(2)_i
\geq C\sum_{j=N}^i\frac{j-2s_{j-1}}{j}d_j
\geq C\sum_{j=N}^i\frac{j-2s_{j}}{j}d_j\\
= C\sum_{j=N}^i\left(1-\frac{2s_{j}}{j}\right)d_j
\geq\frac{C}{2}\sum_{j=N}^i d_j
=\frac{C}{2} s_{N,i}.
\end{multline*}

Now assume $k\geq 3$. We will proceed by induction, so we consider $k=3$ first.
There exist $C$ and $C'$ such that for all $i\geq 3$
\[
a(3)_i
\geq C\sum_{j=3}^i\frac{a(2)_{j-1}}{j}d_j
\geq C'\sum_{j=3}^i s_{N,j-1}\frac{d_j}{j}
= C'\sum_{j=3}^i s_{N,j-1}\sum_{\sigma  \in B_{j,i}^1} f(\sigma).
\]
Now assume that $k\geq 4$. There exists a $C$ such that for all $i$
\begin{multline*}
a(k)_i
\geq C\sum_{j=k}^i\frac{a(k-1)_{j-1}}{j} d_j
\geq C\sum_{j=k}^i\sum_{l=3}^{j-k+3}s_{N,l-1}\sum_{\sigma\in B_{l,j-1}^{k-3}}f(\sigma)\frac{d_j}{j}\\
\geq C\sum_{l=3}^{i-k+3}s_{N,l-1}\sum_{j=l-3+k}^i\frac{d_j}{j}\sum_{\sigma\in B_{l,j-1}^{k-3}}f(\sigma)
=C\sum_{l=3}^{i-k+3}s_{N,l-1}\sum_{\sigma\in B_{l,i}^{k-2}} f(\sigma).
\end{multline*}
\end{proof}

\begin{lemma}\label{lemma:expectation}
Let $k \ge 2$. Then,
\[  \left [\sum_{j=1}^{\infty}d_j t_j^{k-2} = \infty  \right ] \implies  \left [ \lim_{i \rightarrow \infty} a(k)_i = \infty \right ].
\]
\end{lemma}
\begin{proof} The case $k=2$ follows directly from the previous Lemma~\ref{lemma:lowertrees}.
 Hence, assume $k \ge 3$. By Lemma~\ref{lemma:bothdiverge} and the hypothesis we have that 
\[ \sum_{j=1}^{\infty}s_j \sum_{\sigma \in B^{k-2}_j} f(\sigma)=\infty.
\]
Let $N$ be the constant from Lemma~\ref{lemma:lowertrees}. Lemma~\ref{lemma:finitesum} and the fact that $0\leq s_j-s_{N,j-1} \leq N+1$ imply that 
\begin{align*} 0\leq\sum_{j=1}^{\infty}s_j \sum_{\sigma \in B^{k-2}_j} f(\sigma) - \sum_{j=1}^{\infty}s_{N,j-1} \sum_{\sigma \in B^{k-2}_j} f(\sigma) & \le  (N+1)  \sum_{j=1}^{\infty} \sum_{\sigma \in B^{k-2}_j} f(\sigma) \\
& \le   (N+1)  \sum_{j=1}^{\infty} \sum_{\sigma \in A^{k-2}_j} f(\sigma) < \infty.
\end{align*}
Hence we have that $\sum_{j=1}^{\infty}s_{N,j-1} \sum_{\sigma \in B^{k-2}_j} f(\sigma) =\infty$.\\
To complete the proof, we choose $M>0$. Then there is $i_0$ such that 
\[\sum_{j=3}^{i_0+3-k}s_{N,j-1} \sum_{\sigma \in B^{k-2}_j} f(\sigma) > \frac{M+1}{C},
\]
where $C$ is the constant from Lemma~\ref{lemma:lowertrees}.
Now for each $1 \le j \le i_0+3-k$, there is $i_j$ such that   
\[\sum_{\sigma \in B^{k-2}_{j,i_j} } f(\sigma)\ge \frac{-1}{C(i_0+3)(s_{i_0+3})} +\sum_{\sigma \in B^{k-2}_j} f(\sigma).
\]
Now for all $i \ge \max \{i_0, i_1,\ldots i_{i_0+3-k}\}$ we have that
\begin{multline*}
a(k)_i
\ge C \sum_{j=3}^{i+3-k} s_{N,j-1} \sum _{\sigma \in B^{k-2}_{j, i}} f(\sigma)
\ge C \sum_{j=3}^{i_0+3-k} s_{N,j-1} \sum _{\sigma \in B^{k-2}_{j, i_j}} f(\sigma)\\
\ge C \sum_{j=3}^{i_0+3-k} s_{N,j-1} \left (\frac{-1}{C(i_0+3)(s_{i_0+3})} +\sum_{\sigma \in B^{k-2}_j} f(\sigma) \right )\\
\ge -1+ C \sum_{j=3}^{i_0+3-k} s_{N,j-1} \sum_{\sigma \in B^{k-2}_j}f(\sigma)
> -1 + C \frac{M+1}{C}
= M,
\end{multline*}
completing the proof.
\end{proof}

Recall that a \emph{ray} is a one-way infinite path, i.e.\ a sequence of vertices $u_1,u_2,\ldots$ in a graph such that $u_i\sim u_{i+1}$ for all $i$.

\begin{theorem}\label{thm:noray}
Suppose $\sum d_i/i<\infty$. Then the process a.s.\ generates a graph with no
ray.
\end{theorem}
\begin{proof}
The hypothesis implies that $t_n\to 0$, so there exists a positive integer $N$
such that for all $i\geq N$, $t_i<1$. Fix $i\geq N$. For $k\geq 2$, let $E_k$
denote the event that the graph will contain a path of length $k$ that starts
at the vertex $v_i$ and for all vertex $v_j$ on the path we have $j\geq i$.

\begin{multline*}
\Pr(E_k)
= \sum_{\sigma\in B_{i}^{k+1}}\Pr(v_{\sigma_{l+1}}\sim v_{\sigma_l}\text{ for
	all }1\leq l \le k)\\
= \sum_{\sigma\in
	B_{i}^{k+1}}\frac{d_{\sigma_2}}{\sigma_2}\cdots\frac{d_{\sigma_{k+1}}}{\sigma_{k+1}}
\le \sum_{\sigma\in D_{i}^{k}}f(\sigma)
=t_{i}^{k}.
\end{multline*}
Hence $\lim_{k\to\infty}\Pr(E_k)=0$. We conclude that the probability that a
ray emanates from the vertex $v_i$ such that every vertex of the ray is beyond
$v_i$ is $0$. Then a.s., for all $i\geq N$, such ray does not exist.

It is easy to see that if the graph had a ray, it would also have a ray
$v_{i_0}v_{i_1}\ldots$ such that $i_0\geq N$, and for all $j\in\mathbb{N}$,
$i_j\geq i_0$. But we have just seen that the probability of that is $0$.
\end{proof}

\subsection{Proof of Theorem~\ref{thm:zo}}
Part \ref{thmzo:dense}) is an immediate consequence of
Theorem~\ref{thm:degrees} in the special case of zero-one sequences.

For part \ref{thmzo:sparse}), let $m<k$; Lemma~\ref{lemma:expectation} implies
that $a(m)_i\to\infty$, that is the expected number of components of size $m$
is infinity. But we will prove a stronger statement, namely that a.s.\ there
are infinitely many components of size $m$. We will do this in two steps. First,
we will show that a.s.\ infinitely many components of size $m$ are created. Then, using this fact we will show that a.s.\ the final graph has infinitely many components of size $m$.

Let us proceed to show that infinitely many components of size $m$ are
created. If $m=1$, then this follows from the fact that the sequence $\{d_i\}$
contains infinitely many zeros. Let $m\geq 2$. Let $N_j$ be the random
variable that counts the number of components of size $m-1$ spanned by
vertices $\{v_0,\ldots,v_j\}$. Let $E_j$ be the event that a
component of size $m$ is created at step $j$ from vertex $v_j$.  Then,
\begin{multline*}
	\Pr(E_j)=\sum_{\ell=0}^\infty\Pr(N_{j-1}=\ell)\frac{ \ell (m-1)d_j}{j}\\
	=(m-1)\frac{d_j}{j}\sum_{\ell=0}^\infty\ell\Pr(N_{j-1}=\ell)=(m-1)\frac{d_j}{j}a(m-1)_{j-1}.
\end{multline*}

By Lemma~\ref{lemma:recursion},
there exists a constant $C$, such that for all $i\geq 1$,
\[
a(m)_i\leq C\sum_{j=1}^i\frac{a(m-1)_{j-1}}{j}d_j.
\]
Since $a(m)_i\to\infty$, we conclude that
\[
\sum_{j=1}^\infty\Pr(E_j)=
(m-1)\sum_{j=1}^\infty\frac{d_j}{j}a(m-1)_{j-1}=\infty.
\]
To show that a.s.\ infinitely many components of size $m$ are created, it suffices to show that for all $i_0 \in {\mathbb N}$, $\Pr(\cup_{j=i_0}^{\infty} E_j) =1$. We will do this using the following theorem, sometimes referred to as the counterpart of the Borel-Cantelli lemma \cite{Bruss-80}.

\begin{lemma}\label{lemma:BC}
	Let $A_1,A_2,\ldots$ be a sequence of events such that $A_i \subseteq A_{i+1}$. Then,\[
	\Pr \left (\cup_{j=1}^\infty A_j \right) =1 \ \ \   \iff \ \ \ \sum_{j=1 }^{\infty} \Pr(A_{j+1}|A_j^c) = \infty.\]
\end{lemma}

We apply the above lemma to the situation where $A_j = \cup_{i=i_0}^{i_0+j-1} E_i$. We note that for all $j \ge 1$, we have that 
\[ \Pr(A_{{j+1}}|A_{j}^c) \ge \Pr(E_{j+i_0}),
\]
implying that \[  \sum_{j=i_0 }^{\infty} \Pr(A_{j+1}|A_{j}^c) = \infty. \]
By Lemma~\ref{lemma:BC} and the fact that $\cup_{i=1}^{\infty}A_i = \cup_{i=i_0}^{\infty} E_i$, we have that 
$\Pr (\cup_{j=i_0}^\infty E_i) =1$.

We next show that a.s.\ the final graph has infinitely many components of size $m$. It suffices to show that for every $i \in {\mathbb N}$ and $\varepsilon >0$, the probability that the final graph contains a component of size $m$ containing a vertex $v_j$, $j \ge i$, has probability greater than $1-\varepsilon$. Indeed, this is the case as it implies that for all $i$, with probability one, a component of size $m$ with a vertex $v_j$, $j \ge i$ will exist. Then, taking intersection over all $i$'s, we obtain the desired result with probability one. 

If a component of size $m$ is created at vertex $v_j$, the
probability that it will not be destroyed is
\[
q_j=\prod_{\ell=j+1}^\infty\left(1-\frac{d_{\ell} m}{\ell}\right).
\]
As by hypothesis $\sum \frac{d_{\ell} }{\ell}  = \infty$, we have that $q_j\to 1$ as $j\to\infty$. Since a.s.\ $E_j$ occurs for infinitely many $j$'s, we have that for sufficient large $j$ we create a component of size $m$ without later destroying it with probability at least $1 -\varepsilon$.

For any $m'\geq k$, Lemma~\ref{lemma:upbd} shows that the expected number of
components of size $m'$ is finite, therefore a.s.\ there are finitely many
components of size $m'$.

There are two things that remain to be proven to finish the proof of the theorem.
First that if $m<k$, and $T$ is a tree with $|T|=m$, then a.s.\ there are
infinitely many components of the graph isomorphic to $T$. From the argument
above, we have that a.s.\ there are infinitely many components of size $m$, and
also, if $C$ is a component, $\Pr(C\cong T\ |\ |C|=m)>0$, so the statement
follows.

The second thing is that a.s.\ there is no infinite component of the graph.
From Theorem~\ref{thm:degrees}, we know that a.s.\ each vertex is of finite
degree (i.e.\ the graph is locally finite). Every locally finite connected
infinite
graph contains a ray (see e.g.\ Proposition 8.2.1.\ in \cite{Die-GT}). So
Theorem~\ref{thm:noray} finishes the proof.
\qed

\begin{corollary}
If $\sum_{i=1}^\infty\frac{d_i s_i}{i}<\infty$, then the space has
infinitely many atoms, and all of them are of the form $F\cup F_1\cup F_2$ where $F$ is a
finite forest.
\end{corollary}
\begin{proof}
\[
\sum_{j=1}^\infty d_j t_{j+1}
\leq\sum_{j=1}^\infty d_j t_j
=\sum_{j=1}^\infty d_j\sum_{i=j}^\infty\frac{d_i}{i}
=\sum_{i=1}^\infty \frac{d_i}{i}\sum_{j=1}^i d_j
=\sum_{i=1}^\infty\frac{d_i s_i}{i}
<\infty,
\]
so Theorem~\ref{thm:zo} implies the statement.
\end{proof}
The following theorem is a natural analogue of Corollary~\ref{cor:dblrnd1}.
\begin{theorem}\label{thm:dblrnd2}
Fix $0<p<1$, and consider the double random process with $d_i=1$ with
probability $p$, otherwise $d_i=0$ for $i>0$. The process a.s.\ generates infinitely many
copies of $\omega$-trees.
\end{theorem}

\begin{proof}
We will prove that the hypotheses of Theorem~\ref{thm:zo}~(\ref{thmzo:dense}) are satisfied a.s. Let
$X_n=\sum_{i=1}^n d_i/i$. On one hand $\mu_n:=E[X_n]\geq p\ln n$. On the other hand,
\[
\sigma^2_n:=\Var[X_n]=\sum_{i=1}^n\Var\left[\frac{d_i}{i}\right]\leq\sum_{i=1}^n\frac{p-p^2}{i^2}\leq 2.
\]
Fix $M>0$. Using Chebyshev's inequality, if $\mu_n>M$,
\[
\Pr[X_n\leq M]\leq\Pr[|X_n-\mu_n|\geq
\mu_n-M]\leq\frac{\sigma_n^2}{(\mu_n-M)^2}\leq\frac{2}{(p\ln n-M)^2}\to 0.
\]
Hence, a.s.\ $X_n\to\infty$.
\end{proof}

\section*{Acknowledgement}
The authors are indebted to the referee for suggesting substantial changes which improved the paper.

\bibliographystyle{amsplain}

\end{document}